\documentclass[a4paper,10pt]{article}
\usepackage[utf8]{inputenc}
\usepackage{amsmath}		
\usepackage{amssymb}		
\usepackage{euscript}
\usepackage[amsmath,thmmarks]{ntheorem}	
\usepackage{tikz-cd}
\usepackage[colorlinks=true,citecolor=blue]{hyperref}
\usepackage{mathtools}
\usepackage{bookmark}
\usepackage[all]{xy}
 \usepackage[top=3.5cm, bottom=3.7cm, left=3.7cm, right=3.7cm, includehead,
 includefoot,headsep=4mm]{geometry}
\usepackage{enumitem}
\setlist{topsep=0mm,parsep=0mm,itemsep=.5mm}

\mathtoolsset{
showonlyrefs=true 
}
\DeclareMathAlphabet{\mathpzc}{OT1}{pzc}{m}{it}

\newcommand{\pzf}{\mathpzc{F}}

\newcommand{\catset}{{\mathbf{Set}}}

\newcommand{\catiord}{\mathbf{IOrd}}

\newcommand{\catord}{\mathbf{Ord}}

\newcommand{\catc}{\mathbb{C}}

\newcommand{\catx}{\mathbb{X}}

\newcommand{\pcaa}{\mathcal{A}}

\newcommand{\trip}{{\EuScript{P}}}

\newcommand{\catrt}{\mathbf{RT}}

\newcommand{\comk}{\mathsf{k}}				

\newcommand{\comc}{\mathsf{c}}
\newcommand{\comp}{\mathsf{p}}

\newcommand{\compz}{\mathsf{p_0}}
\newcommand{\compo}{\mathsf{p_1}}

\newcommand{\msep}{\mathrel|\,}		

\newcommand{\op}{^\mathsf{op}}				
\newcommand{\id}{\mathrm{id}}				

\newcommand{\ap}{\mathclose\cdot}			
\newcommand{\qdot}{\,.\,}				

\newcommand{\setof}[2]{\{#1\msep #2\}}

\newcommand{\N}{\mathbb{N}}

\newcommand{\mono}{\rightarrowtail}
\newcommand{\monol}{\leftarrowtail}
\newcommand{\incl}{\hookrightarrow}

\newcommand{\epi}{\twoheadrightarrow}

\newcommand{\epicart}{\xymatrix@1@-3mm{{}\epicart[r]&{}}}
\newcommand{\covercart}{\xymatrix@1@-3mm{{}\epicart[r]&{}}}

\makeatletter
\providecommand*{\twoheadrightarrowfill@}{%
  \arrowfill@\relbar\relbar\twoheadrightarrow
}
\providecommand*{\twoheadleftarrowfill@}{%
  \arrowfill@\twoheadleftarrow\relbar\relbar
}
\providecommand*{\xepi}[2][]{%
  \ext@arrow 0579\twoheadrightarrowfill@{#1}{#2}%
}
\providecommand*{\xepil}[2][]{%
  \ext@arrow 5097\twoheadleftarrowfill@{#1}{#2}%
}
\makeatother

\makeatletter
\newcommand*{\relrelbarsep}{.386ex}
\newcommand*{\relrelbar}{%
  \mathrel{%
    \mathpalette\@relrelbar\relrelbarsep
  }%
}
\newcommand*{\@relrelbar}[2]{%
  \raise#2\hbox to 0pt{$\m@th#1\relbar$\hss}%
  \lower#2\hbox{$\m@th#1\relbar$}%
}
\providecommand*{\rightrightarrowsfill@}{%
  \arrowfill@\relrelbar\relrelbar\rightrightarrows
}
\providecommand*{\leftleftarrowsfill@}{%
  \arrowfill@\leftleftarrows\relrelbar\relrelbar
}
\providecommand*{\xrightrightarrows}[2][]{%
  \ext@arrow 0359\rightrightarrowsfill@{#1}{#2}%
}
\providecommand*{\xleftleftarrows}[2][]{%
  \ext@arrow 3095\leftleftarrowsfill@{#1}{#2}%
}
\makeatother

\newcommand{\tensor}{\otimes}

\newcommand{\abs}[1]{\lvert#1\rvert}

\newcommand{\ve}{\varepsilon}

\newcommand{\qtext}[1]{\quad\text{#1}\quad}
\newcommand{\qqtext}[1]{\qquad\text{#1}\qquad}

\newcommand{\imp}{\Rightarrow}

\newcommand{\famf}{\mathrm{fam}}

\newcommand{\pullbackcorner}[1][dr]{\save*!/#1-1.2pc/#1:(-1,1)@^{|-}\restore}

\newcommand{\adj}{\dashv}

\newcommand{\pto}{\rightharpoonup}

\newcommand{\proj}{\mathbf{Proj}}

\newcommand{\catlex}{\mathbf{Lex}}

\newcommand{\catex}{\mathbf{Ex}}

\newcommand{\exlex}[1]{{#1}_{\mathrm{ex}}}

\newcommand{\adjr}[2]{
\ar@/_6pt/[r]_{#1}
\ar@{}[r]|\top
\ar@{<-}@/^6pt/[r]^{#2}
}
\newcommand{\adjrr}[2]{
\ar@/_6pt/[rr]_{#1}
\ar@{}[rr]|\top
\ar@{<-}@/^6pt/[rr]^{#2}
}
\newcommand{\adjl}[2]{
\ar@/^6pt/[l]^{#1}
\ar@{}[l]|\top
\ar@{<-}@/_6pt/[l]_{#2}
}
\newcommand{\adjd}[2]{
\ar@/_6pt/[d]_{#1}
\ar@{}[d]|\dashv
\ar@{<-}@/^6pt/[d]^{#2}
}
\newcommand{\adju}[2]{
\ar@/^6pt/[u]^{#1}
\ar@{}[u]|\dashv
\ar@{<-}@/_6pt/[u]_{#2}
}
\newcommand{\dadj}[4]{ 
\ar@/^6pt/[#3]^{#1}
\ar@{}[#3]|{#4}
\ar@{<-}@/_6pt/[#3]_{#2}
}

\newcommand{\fifc}{\EuScript{C}}

\newcommand{\fifp}{\EuScript{P}}

\newcommand{\defined}{\mathclose\downarrow}

\newcommand{\les}{\preceq}
\newcommand{\ges}{\succeq}

\newcommand{\pcakone}{\mathcal{K}_1}

\newcommand{\conflict}{\ar@{~}}

\def\signed #1{{\leavevmode\unskip\nobreak\hfil\penalty50\hskip2em
  \hbox{}\nobreak\hfil(#1)%
  \parfillskip=0pt \finalhyphendemerits=0 \endgraf}}

\newsavebox\mybox

\newcommand{\gcons}{\textstyle\int\!}

\newcommand{\inv}{^{-1}}

\newcommand{\setord}{\catset\op\to\catord}

\newcommand{\pasm}{\mathbf{PAsm}}

\newcommand{\subs}{\subseteq}

\newcommand{\sups}{\supseteq}

\newcommand{\catdco}{\mathbf{DCO}}
\newcommand{\daf}{(A,\pzf_A)}
\newcommand{\dbg}{(B,\pzf_B)}

\newcommand{\funa}{{\pzf_A}}
\newcommand{\funb}{{\pzf_B}}

\newcommand{\const}{\mathsf{c}}
\newcommand{\att}{@} 
\newcommand{\pcad}{(\pcaa,\cdot)}

\newcommand{\tatu}[2]{\langle #1,#2\rangle}

\newcommand{\lar}{\ar@{-|>}} 
\newcommand{\rar}{\ar@{ |>->}}

\makeatletter
\newcommand{\dotminus}{\mathbin{\text{\@dotminus}}}

\newcommand{\@dotminus}{%
  \ooalign{\hidewidth\raise1ex\hbox{.}\hidewidth\cr$\m@th-$\cr}%
}
\makeatother

\newcommand{\speq}{\;=\;}

\newcommand{\ie}{i.e.\ }
\newcommand{\eg}{e.g.\ }

\newcommand{\wlcc}{w.l.c.c.\ }

\newcommand{\xotn}{x_1, \dots, x_n}

\newtheorem{theorem}{Theorem}[section]
\newtheorem{proposition}[theorem]{Proposition}
\newtheorem{lemma}[theorem]{Lemma}

\newtheorem{corollary}[theorem]{Corollary}

\theorembodyfont{\upshape}
\theoremsymbol{\ensuremath{\diamondsuit}}

\newtheorem{definition}[theorem]{Definition}

\theorembodyfont{\normalfont}

\newtheorem{remark}[theorem]{Remark}

\newtheorem{remarks}[theorem]{Remarks}
\newtheorem{example}[theorem]{Example}

\newtheorem{conventions}[theorem]{Conventions}

\theoremsymbol{\ensuremath{_\blacksquare}}
\theoremstyle{nonumberplain}
\theoremheaderfont{\itshape}
\newtheorem{proof}{Proof.}

\qedsymbol{\ensuremath{_\blacksquare}}

\theorembodyfont{\itshape}
\theoremheaderfont{\normalfont\bfseries}

\numberwithin{equation}{section}		

\newcommand{\depi}{\ar@{->>}}
\newcommand{\emar}{\ar@{}}

\title{Characterizing partitioned assemblies \\ and realizability toposes}
\author{Jonas Frey}

\begin{document}

\maketitle

\begin{abstract}
  We give simple characterizations of the category $\pasm(\pcaa)$ of
  \emph{partitioned assemblies}, and of the \emph{realizability topos}
  $\catrt(\pcaa)$ over a partial combinatory algebra $\pcaa$.
This answers the question for an `extensional characterization' of realizability
toposes.
\end{abstract}

\section{Introduction}

Realizability toposes $\catrt\pcad$ over partial combinatory algebras (PCAs) 
$\pcad$ were introduced in 1980 by Hyland, Johnstone and Pitts~\cite{hjp80} as 
categories of internal partial equivalence relations in certain indexed 
preorders which they called \emph{triposes}. 
In 1990, Robinson and Rosolini~\cite{robinson1990colimit} showed that Hyland's 
\emph{effective topos} $\catrt(\pcakone,\cdot)$~\cite{hyland82} -- the most well 
known realizability topos, constructed from the PCA
$(\pcakone,\cdot)$ known as \emph{first Kleene algebra} -- 
is the \emph{exact completion}~\cite{carboni1982free} of the category 
$\pasm(\pcakone,\cdot)$ of \emph{partitioned assemblies} over 
$(\pcakone,\cdot)$, a result that easily generalizes to realizability toposes
over arbitrary PCAs.

The present article gives an extensional characterization of categories 
of partitioned assemblies over PCAs (Theorem~\ref{thm:char-pasm}, 
Corollary~\ref{cor:char-pasm}), which by means of (the generalization of) 
Robinson and Rosolini's result yields a characterization of realizability 
toposes over PCAs (Proposition~\ref{prop:char-rt}, Theorem~\ref{thm:char-rt'}), 
and on the other hand justifies the definition of PCA by reconstructing it from 
abstract concepts. The notion of PCA that drops out of our reconstruction is a 
bit more general than the classical notion used \eg 
in~\cite{hjp80, vanoosten2008realizability}, but seems to have been adapted as 
standard in more recent literature~\cite{stekelenburg2013realizability,
johnstone2013,faber2014more,faber2016effective}. 
For the purpose of characterizing partitioned assemblies and realizability 
toposes the distinction is immaterial by a result of Faber and van Oosten which 
says that any PCA in the more general sense can be 
`strictified'~\cite[Theorem~5.8]{faber2016effective}.

For the characterization of partitioned assemblies we adapt techniques from 
Hofstra's analysis of \emph{ordered partial combinatory algebras} (OPCAs) in 
terms of his \emph{basic combinatory objects} 
(BCOs), which are partial orders equipped with a class of partial endofunctions 
subject to certain axioms. Every OPCA gives rise to a BCO, and Hofstra 
characterizes (inclusions of) OPCAs among 
BCOs~\cite[Propositions~6.5, 6.6]{hofstra2006all}. Moreover he gives a new 
perspective on partitioned assemblies by embedding BCOs into indexed preorders 
in such a way that the total category of the indexed preorder associated to a 
BCO arising from an OPCA is precisely the category of partitioned assemblies 
over the OPCA.

Since we are only interested in non-ordered PCAs we adapt Hofstra's analysis by 
restricting attention to non-ordered BCOs, which we call 
\emph{discrete combinatory objects} (DCOs). This way we obtain a 
characterization of PCAs among DCOs in analogy to Hofstra's characterization of 
OPCAs among BCOs (Corollary~\ref{cor:char-pca}), and moreover there is an easy 
description of the indexed preorders arising from DCOs 
(Proposition~\ref{prop:dcogenpred}), so that in combination we can identify PCAs with 
a class of indexed preorders. To obtain a characterization of categories of 
partitioned assemblies it remains to characterize the total categories of these 
indexed preorders, which we do by identifying certain properties of these total 
categories which allow to reconstruct the indexed preorders, and thus the PCAs. 

A crucial insight here is that whereas indexed preorders can generally not be 
reconstructed from their total categories, those which arise from PCAs 
\emph{can}, since in this case the forgetful functor from the total category 
(which is the necessary additional datum to reconstruct the indexed preorder) 
coincides with the global sections functor. To capture this phenomenon we 
introduce the notions of \emph{shallow indexed preorder} 
(Section~\ref{sec:shallow-cartesian}) and \emph{local category} 
(Section~\ref{sec:local-categories}).

\subsection{Related work}

The present work, and Hofstra's analysis of OPCAs in terms of BCOs, fit into a 
line of work whose general theme is to first generalize the construction of 
realizability models (triposes, toposes, assemblies) from PCAs to a more general 
class of `combinatory structures' and then to identify conditions on these 
structures which ensure certain logical properties of the 
model~\cite{lietz2002impredicativity,birkedal2000general,robinson2001abstract}.
In particular, Robinson and Rosolini~\cite[Corollary~2]{robinson2001abstract} 
give necessary and sufficient conditions for the realizability categories over a 
specific type of categories of partial maps to be locally cartesian closed, and 
Lietz and Streicher~\cite[Theorem~4.2]{lietz2002impredicativity} -- and in a 
similar form Birkedal~\cite[Corollary~5.3]{birkedal2000general} -- show that 
realizability categories over certain typed combinatory structures are toposes 
if and only if the structures have a `universal type'.

The main novelty of the present work is the \emph{reconstruction} of the 
combinatory structure from the realizability category, which allows an intrinsic characterization without reference to the combinatory structure. This in turn 
leads to concepts related to Menni's axiomatic 
approach~\cite{menni2003characterization,menni2002more} to the study of realizability-like exact completions, see 
Remark~\ref{rem:local}-\ref{item:menni}.

\section{Discrete combinatory objects}\label{sec:dco}
As pointed out above, DCOs are the trivially ordered special case of Hofstra's 
BCOs. In the following (before Definition~\ref{def:discrete-genpred}) we recall 
basic definitions and results, which simplify considerably in the absence of 
ordering. We refer to Hofstra for proofs, but the arguments are straightforward 
and the readers are encouraged to reconstruct them themselves. 
Proposition~\ref{prop:dcogenpred} is new, the author is not aware of an 
analogous result for BCOs. Before introducing DCOs we establish some conventions 
concerning partial functions. 

\begin{conventions}[Partial functions and partial terms]
Throughout the rest of the article we perform calculations with partial 
functions both unapplied/compositionally, and applied/applicatively. 

In the unapplied form we view partial functions as represented by their graphs, 
which we compare via subset inclusion and compose like functions and relations.
The cartesian product of sets and functions extends to a tensor product on 
partial functions by setting
\[
 f\times g\;=\;\setof{(a,b,c,d)}{(a,c)\in f,(b,d)\in g} \;:\; A\times B\pto C\times D
\]
for $ f:A\pto C$ and $ g:B\pto D$. 
The pairing operation $\langle-,-\rangle$, given by
\[
\langle f, h\rangle =( f\times h)\circ\delta_A\qtext{for}  f:A\pto B \qtext{and}  h:A\pto C,
\]
satisfies the usual equations 
$
( k\times l)\circ\langle f, h\rangle = \langle k\circ f, l\circ h\rangle 
$
and 
$
\langle f, h\rangle\circ m=\langle f\circ m, h\circ m\rangle
$, as well as the inclusions $\pi_1\circ\langle f, h\rangle\subs f$ and 
$\pi_2\circ\langle f, h\rangle\subs h$, which become equalities whenever the 
\emph{eliminated} term is total.

When reasoning applicatively -- \ie with partially defined terms -- the 
statement $t\defined$ asserts that the term $t$ is defined, and $s=t$ states 
that both $s$ and $t$ are defined and equal. Equality of graphs is represented 
by the expression $s\cong t$ which is a shorthand for 
$s\defined\vee t\defined\imp s=t$, and inclusion of graphs is expressed as 
$s\les t$, which is a shorthand for $s\defined\imp s=t$.
\end{conventions}
\begin{definition}\label{def:dco}
\begin{enumerate}
\item \label{def:dco-dco}
A  \emph{discrete combinatory object} (DCO) is a pair 
$\daf$ where $A$ is a set and $\funa$ is a set of partial endofunctions on $A$
which contains $\id_A$, and such that for all $\alpha,\beta\in\funa$ there 
exists a $\gamma\in\funa$ with $\beta\circ\alpha\subs\gamma$.
\item A \emph{DCO morphism} from $\daf$ to $\dbg$ is a function
$f:A\to B$ such that for all $\alpha\in\funa$ there exists a $\beta\in\funb$
with $f\circ\alpha\subs\beta\circ f$.
\item For DCO morphisms $f,g:\daf\to\dbg$, we define
$f\leq g$ if there exists a $\beta\in \funb$ with $\beta\circ f = g$. In this 
case, we call $\beta$ a \emph{realizer} of the inequality $f\leq g$.
\end{enumerate}
\end{definition}
It is easy to see that DCO morphisms compose, and that the relation defined in 3 
is reflexive and transitive and preserved by composition on both sides. Thus, 
DCOs form a locally ordered category $\catdco$, which furthermore has a terminal 
object $1=(1,\{\id\})$ and binary 2-products given by
\begin{align*}
\daf\times \dbg&=(A\times B,\funa\tensor \funb)\\
\text{where }\quad 
\funa\tensor\funb&=
\setof{\alpha\times \beta}{\alpha\in \funa,\beta\in \funb}.
\end{align*}
Analogous statements for BCOs can be found in \cite[Section~2]{hofstra2006all}.

Given a DCO $\daf$ and a set $I$, we define a preorder on functions 
$\varphi,\psi:I\to A$ by setting $\varphi\leq\psi$ if there exists an 
$\alpha\in\funa$ with $\alpha\circ\varphi=\psi$. Again, we call $\alpha$ a 
\emph{realizer} of the inequality in this situation. The construction is
contravariantly functorial in $I$ and thus gives rise to a \emph{split indexed
  preorder}, \ie a contravariant functor 
\[\famf\daf:\setord,\qquad I\mapsto (A^I,\leq)\] 
from $\catset$ into the locally ordered 
category $\catord$ of preorders and monotone maps. 
We call $\famf\daf$ the
\emph{family fibration} of $\daf$.
The assignment $\daf\mapsto\famf\daf$ extends to a 2-functor
\[
\famf(-):\catdco\to\catiord
\]
into the locally ordered category $\catiord$ of indexed preorders and
pseudo-natural transformations. This 2-functor is a local
equivalence\footnote{I.e.\ the monotone map
  $\catdco(\daf,\dbg)\to\catiord(\famf\daf,\famf\dbg)$ is an equivalence of
  preorders for all DCOs $\daf,\dbg$.} -- as Hofstra shows for 
  BCOs~\cite[Proposition~3.1]{hofstra2006all} -- and is easily seen to
preserve finite 2-products.  Moreover, there is a
straightforward characterization of the essential image of $\famf(-)$, using the
following definition. 
\begin{definition}\label{def:discrete-genpred}
\NoEndMark
Let $\fifp:\setord$, $A\in\catset$, and $\mu\in\trip(A)$.
\begin{enumerate}
\item \label{def:discrete-genpred-d} $\mu$ is called
\emph{discrete}, if for any span $I\xepil{e}J\xrightarrow{f}A$ of functions
with $e$ surjective, and any $\varphi\in\fifp(I)$ such that
$e^*\varphi\leq f^*\mu$, there exists a (necessarily unique) $h:I\to A$
such that $he=f$ and $\varphi\leq h^*\mu$.
\item\label{def:discrete-genpred-g}
$\mu$ is called a \emph{generic predicate}, if for any 
set $I$ and any $\varphi\in\fifp(I)$ there exists a (not necessarily unique) 
function $f: I\to A$ with $\varphi\cong f^*\mu$.\hfill\definitionSymbol
\end{enumerate}
\end{definition}
\begin{proposition}\label{prop:dcogenpred}
An indexed preorder $\fifp$ is in the essential image of $\famf(-)$ if 
and only if it has a discrete generic predicate.
\end{proposition}
\begin{proof}
First, let $\daf$ be a DCO. The identity map $\id_A\in A^A$ is a generic 
predicate of $\famf\daf$ since every predicate $\varphi:I\to A$ can be 
represented as $\varphi^*(\id_A)=\id_A\circ\varphi$. 

To show that $\id_A$ is discrete, assume that $e^*(\varphi)\leq f^*(\id_A)$ for 
a span $I\xepil{e}J\xrightarrow{f}A$ with surjective $e$. By definition of the
order on $A^I$, there exists an $\alpha\in\funa$ with $\alpha\circ\varphi\circ e
= f$, and the required mediator $g:I\to A$ is given by $\alpha\circ\varphi$. 

Conversely, let $\fifp$ be an indexed preorder with discrete generic predicate
$\mu\in\fifp(A)$. Every partial function $\alpha\subs A\times A$ gives rise to 
a 
span $A\stackrel{d
}{\monol}\alpha\xrightarrow{a} A$, and the partial functions $\alpha$ satisfying
$d^*(\mu)\leq a^*(\mu)$ form a DCO structure $\funa$ on $A$. The assignment
$(f:I\to A)\mapsto f^*(\mu)$ defines an indexed monotone map
$\Phi:\famf\daf\to\fifp$, which is essentially surjective since $\fifp$ has a
generic predicate. To show that $\Phi$ is fiberwise order-reflecting, let 
$f, g:I\to A$ such that $f^*(\mu)\leq g^*(\mu)$ and consider the 
diagram
\[
\begin{tikzcd}
 &I\ar[dr,"g"]\ar[d,two heads,"e"] \ar[ld,"f"']\\
 A&U\ar[r,dashed,"h"]\ar[l,tail,"m"]& A
\end{tikzcd}
\]
where $m\circ e$ is a surjective/injective factorization of $f$.
Since $e^*(m^*\mu)\leq g^*\mu$ and $\mu$ is discrete, there exists an $h:U\to A$
with $he= g$ and $m^*\mu\leq h^*\mu$. The span $(m,h)$ constitutes a 
partial function in $\funa$ witnessing the inequality $f\leq g$ in $\famf\daf$.
\end{proof}

\begin{remark}[Saturation] Calling a DCO $\daf$ \emph{saturated} if $\funa$
	is a lower set in $P(A\times A)$, it is easy to see that every DCO is 
	isomorphic to a saturated one -- its \emph{saturation} -- obtained by 
	down-closing $\funa$ in $P(A\times A)$. While the second condition 
	in \hbox{Definition~\ref{def:dco}-\ref{def:dco-dco}} can be read as a weak 
	closure 
	under composition of $\funa$, it turns out that if $\daf$ is saturated then 
	$\funa$ is even closed under composition in the strong sense. The DCOs 
	constructed
	from indexed preorders in the proof of Proposition~\ref{prop:dcogenpred}
	are always saturated, but the DCOs defined from
	PCAs below in Definition~\ref{def:pca-dco} are not, and
	saturation is not preserved by the product operation given after 
	Definition~\ref{def:dco} (though a saturated product can of course be 
	obtained by down-closing).
\end{remark}

\subsection{Shallow and cartesian DCOs}\label{sec:shallow-cartesian}

We call an indexed preorder $\fifc:\setord$ \emph{shallow} if 
$\fifc(1)\simeq 1$, \ie the fiber over the singleton is equivalent to the 
terminal preorder. A \emph{DCO} $\daf$ is called shallow, if $\famf\daf$ is a 
shallow indexed preorder.

A \emph{cartesian DCO} is a cartesian object in the locally ordered category $\catdco$, i.e.\ a DCO $\daf$ 
such that the maps 
\[{!_A}:\daf\to 1 \qtext{and} \delta_A:\daf\to \daf\times \daf\]
have right adjoints
\[
\top:1\to\daf \qtext{and}\wedge:\daf\times\daf\to\daf.
\]
Since $\famf(-)$ is a local equivalence and preserves finite 2-products, $\daf$ is a cartesian DCO if and only if $\famf\daf$ is a \emph{cartesian indexed preorder} -- \ie a cartesian object in $\catiord$ -- which in turn
is equivalent to $\famf\daf$ having fiberwise
finite meets which are stable under reindexing. In particular, given a cartesian DCO $\daf$ and a set $I$, binary meets and a greatest element in $\famf\daf(I)=(A^I,\leq)$ are given by
\begin{align}
	\top_I&=(I\to 1\xrightarrow{\top}A)&&\text{and}\\
	\varphi\wedge_I\psi&=(I\xrightarrow{\langle\varphi,\psi\rangle}A\times A\xrightarrow{\wedge}A)&&\text{for }\varphi,\psi:I\to A.
\end{align}
%
The following lemma characterizes shallow cartesian DCOs.
\begin{lemma}\label{lem:shallow-cartesian-dco} 
A DCO  $\daf$ is shallow and cartesian if and only if
  \begin{enumerate}
  \item $A$ is inhabited and $\funa$ contains all constant functions $\comc_a$
    for $a\in A$, and
  \item there exists a function $\wedge:A\times A\to A$ and    
    $\lambda,\rho\in\funa$ such that
\begin{enumerate}
\item 
for all $a,b\in A$ we have $\lambda(a\wedge b)=a$ and $\rho(a\wedge b)=b$, and
\item 
for all $\alpha,\beta\in\funa$ there exists a $\gamma\in\funa$ with 
$\wedge\circ\langle\alpha,\beta\rangle\subs\gamma$.
\end{enumerate}
  \end{enumerate}
\end{lemma}
\begin{proof}
  Condition 1 is easily been seen to be equivalent to $\daf$ being shallow and
  $\daf\to 1$ having a right adjoint. In the following we show that 2 is
  equivalent to the existence of a right adjoint to
  $\delta_A:\daf\to \daf\times \daf$.
  
Given a right adjoint $\wedge$ to $\delta_A$, we take $\lambda$ and $\rho$ to be
realizers of $\wedge\leq\pi_1$ and $\wedge\leq\pi_2$, respectively, so that
(a) is satisfied.
Given $\alpha,\beta\in\funa$, let $U\subs A$ be the intersection of the domains
of $\alpha$ and $\beta$ (which is precisely the domain of
$\langle\alpha, \beta\rangle$), and let $\iota,\varphi,\psi:U\to A$ be
respectively the inclusion and the restrictions of $\alpha$ and $\beta$ to
$U$. Then $\alpha$ and $\beta$ realize $\iota\leq\varphi$ and $\iota\leq\psi$,
and for $\gamma$ a realizer of $\iota\leq\varphi\wedge\psi$ we have
$\gamma\circ\iota=\varphi\wedge\psi$. The claim follows since
$\gamma\circ\iota\subs\gamma$ and
$\varphi\wedge\psi=\wedge\circ\langle\alpha,\beta\rangle$.


Conversely assume that (a) and (b) hold. We show that for any set $I$ and
$\varphi,\psi:I\to A$ the function $\wedge\circ\langle\varphi,\psi\rangle$ is a
meet of $\varphi$ and $\psi$ in $\famf\daf(I)$.  The partial functions $\lambda$
and $\rho$ realize $\wedge\circ\langle\varphi,\psi\rangle\leq\varphi$ and
$\wedge\circ\langle\varphi,\psi\rangle\leq\psi$. Let $\theta : I\to A$ and let
$\alpha$ and $\beta$ be realizers of $\theta\leq\varphi$ and
$\theta\leq\psi$. Let $\gamma\in\funa$ such that
$\gamma\sups\wedge\circ\langle\alpha,\beta\rangle$. Precomposing with $\theta$
on both sides gives
$\gamma\circ\theta\sups\wedge\circ\langle\alpha\circ\theta,
\beta\circ\theta\rangle=\wedge\circ\langle\varphi, \psi\rangle$ and since the
right hand side is total, the two are equal.
\end{proof}
An easy consequence of the lemma is that every non-trivial cartesian DCO is infinite, since $\wedge$ and $\langle\lambda,\rho\rangle$ exhibit $A\times A$ as a retract of $A$.

\subsection{PCAs and functionally complete DCOs}

\begin{definition}
	\begin{enumerate}
		\item 
		A \emph{partial applicative structure (PAS)} is a set $\pcaa$ with a partial 
		binary
		operation \[(-\,\ap\,-):\pcaa\times\pcaa\to\pcaa\]
		called \emph{application}. 
		A \emph{polynomial} over $\pcaa$ is a term $t[\xotn]$
		built up from variables $\xotn$, constants in $\pcaa$, and application. 
		\item
		A \emph{partial combinatory algebra (PCA)} is a PAS $\pcaa$
		such that for every polynomial $t[x_1,\dots,x_{n+1}]$ there exists an
		$e\in \pcaa$ such that
		\begin{equation*}
		e\ap a_1\ap\dots\ap a_{n}\defined\qqtext{and}
		t[a_1,\dots,a_{n+1}]\les e\ap a_1\ap\dots\ap a_{n+1}
		\end{equation*}
		for all $a_1,\dots,a_{n+1}\in \pcaa$.
	\end{enumerate}
\end{definition}
It follows from the definition that every PCA $\pcad$ contains elements $\comk,\comp,\compz,\compo$
satisfying $\comk\ap a\ap b=a$, $\compz\ap(\comp\ap a\ap b) = a$, and
$\compo\ap(\comp\ap a\ap b)=b$ for all $a,b\in\pcaa$ (in particular,
$\comk \ap a$ and $\comp\ap a \ap b$ are always defined,
see~\cite[Section~1.1]{vanoosten2008realizability}).

As mentioned in the introduction, the above definition of PCA is more general 
than the traditional one; the latter is obtained by replacing the symbol $\les$ 
by $\simeq$.

\begin{definition}\label{def:pca-dco}
	Let $\pcad$ be a PCA. The set $\pzf_\pcaa$ of \emph{computable functions} 
	over $\pcaa$ is the set of partial functions
	\[
	\phi_a:\pcaa\pto\pcaa\qtext{defined by} \phi_a(b)\simeq a\ap b
	\]
	for $a\in\pcaa$.
	The pair $(\pcaa,\pzf_\pcaa)$ is a DCO which we call the \emph{DCO induced by $\pcad$}.
\end{definition}
\begin{lemma}\label{lem:pca-shallow-cartesian}
	The DCO $(\pcaa,\pzf_\pcaa)$ is shallow and cartesian for every PCA $\pcad$.
\end{lemma}
\begin{proof}
	We verify the conditions of Lemma~\ref{lem:shallow-cartesian-dco}. We have
	already seen that $\pcaa$ is inhabited, and constant functions are computable
	using $\comk$. The partial functions $\lambda$ and $\rho$ are given by
	$\phi_\compz$ and $\phi_\compo$, and $\wedge$ is given by
	$a\wedge b=\comp\ap a\ap b$. It remains to show that for all $a,b\in\pcaa$
	there exists a $c\in\pcaa$ with
	$\wedge\circ\langle\phi_a,\phi_b\rangle\subs\phi_c$,
	\ie $\comp\ap(a\ap x)\ap(b\ap x)\les c\ap x$ for all $x\in\pcaa$. Such a $c$
	exists by the definition of PCA.
\end{proof}
To characterize the shallow cartesian DCOs that arise from PCAs, we introduce
the following concept
(adapted from~\cite[Proposition~6.3~(ii)]{hofstra2006all}).
\begin{definition}\label{def:funcompl}
	A cartesian DCO $\daf$ is called \emph{functionally complete}, if there 
	exists an $@\in \funa$ (called the \emph{universal function}) such that for 
	every $\alpha\in \funa$ there exists a \emph{total } $\widetilde{\alpha}\in 
	\funa$ satisfying 
	\begin{equation}\label{eq:def-funcompl}
	\alpha(a\wedge b)\les @(\widetilde{\alpha}(a)\wedge b)
	\end{equation}
	for all $a,b\in A$.
\end{definition}
\begin{remark}
The functional completeness condition above resembles the notion of
\emph{weak partial evaluation} in \cite[Definition~3]{robinson2001abstract}, and
condition (ii) in~\cite[Proposition~6.3]{hofstra2006all}. The `weak closure'
condition in~\cite[Definition~2.3]{birkedal2000general} is stronger, since it
replaces the inclusion~\eqref{eq:def-funcompl} of partial functions by an 
equality.
\end{remark}
\begin{lemma}\label{lem:pca-funcompl}
	DCOs $(\pcaa,\pzf_\pcaa)$ induced by PCAs are functionally complete.
\end{lemma}
\begin{proof}
	Since $\pcaa$ is a PCA
	there exists an 
	$e\in\pcaa$ with $(\compz\ap a)\ap(\compo \ap a)\les e\ap a$ for 
	$a\in\pcaa$, which implies $a\ap b\les e\ap(\comp\ap a\ap b)$ for $a,b\in\pcaa$.
	We define the universal function by $@=\phi_e$. Now given $r\in\pcaa$, 
	there exists a $\widetilde{r}\in\pcaa$
	with $\widetilde{r}\ap a\defined$ and $r\ap(\comp\ap a\ap b)\les 
	\widetilde{r}\ap a\ap b$ for $a,b\in\pcaa$.
	This implies $\phi_r(a\wedge b)\les@(\phi_{\widetilde{r}}(a)\wedge b)$.
\end{proof}
\begin{example}
	The set of primitive recursive functions constitutes a shallow cartesian DCO
	structure on $\N$ which is \emph{not} functionally complete -- the existence
	of a universal primitive recursive function would lead to a contradiction by
	diagonalization.
\end{example}
Any functionally complete DCO $\daf$ gives rise to a
PAS $(A,\cdot)$ where the partial application is given by
\begin{equation}\label{eq:ap-from-fc}
a\ap b \simeq @(a\wedge b)\qquad\text{for }a,b\in A.
\end{equation}
In the following lemma we show that this gives a PCA if $\daf$ is shallow.
\begin{lemma}\label{lem:dco-pca}
	Let $\daf$ be cartesian, shallow, and functionally complete.
	\begin{enumerate}
		\item\label{lem:dco-pca-a} 
		For any polynomial $t[x_1,\dots,x_{n}]$ over the induced PAS $(A,\cdot)$
		there exists an $\alpha\in \funa$ with
		\[
		t[a_1,\dots,a_n]\les \alpha(\top\wedge a_1\wedge\dots\wedge a_n)
		\]
		for $a_1,\dots,a_n\in A$ (by convention $\wedge$ associates to the left).
		\item\label{lem:dco-pca-b} 
		For any $\alpha\in\funa$ and $n > 0$ there exists an $e\in A$ such that
		\[
		e\ap a_1\ap\dots\ap a_{n-1}\defined \qtext{and}
		\alpha(\top\wedge a_1\wedge\dots\wedge a_n)\les e\ap a_1\ap\dots\ap a_n
		\]
		for $a_1,\dots,a_n\in A$.
		\item\label{lem:dco-pca-c} $(A,\ap)$ is a PCA, and $\id_A$ constitutes 
		an isomorphism between $\daf$ and the induced DCO.
	\end{enumerate}
\end{lemma}
\begin{proof}
	The first claim is shown by induction on the structure of 
	$t[x_1,\dots,x_{n}]$.

	If $t[x_1,\dots,x_n]\equiv x_i$ then $\alpha=\rho\circ\lambda^{n-i}$.
	If $t[x_1,\dots,x_n]\equiv a$ for $a\in\pcaa$, then $\alpha$ is given
	by the constant function $\const_a$.
	
	If $t[x_1,\dots,x_n]\equiv u[x_1,\dots,x_n]\ap v[x_1,\dots,x_n]$, then by
	assumption there exist $\alpha,\beta\in \funa$ such that 
	$\alpha(a^*)\ges u[a_1,\dots,a_n]$ and $\beta(a^*)\ges v[a_1,\dots a_n]$ for
	$a_1,\dots,a_n\in A$ and $a^*=\top\wedge a_1\wedge\dots\wedge a_n$.  
	By Lemma~\ref{lem:shallow-cartesian-dco} there exists a $\gamma\in\funa$ 
	such that $\gamma\sups\wedge\circ\tatu{\alpha}{\beta}$, and the calculation
	\begin{equation*}
	(@\circ\gamma)(a^*)
	\;\ges\; @(\wedge(\tatu{\alpha}{\beta}(a^*))
	\;\cong\; \alpha(a^*)\ap\beta(a^*)\\ 
	\;\ges\; u[a_1,\dots,a_n]\ap v[a_1,\dots,a_n]
	\end{equation*}
	shows that $@\circ\gamma$ has the required property.
	
	For the second claim set $\alpha_0=\alpha$ and 
	$\alpha_{i+1}=\widetilde{\alpha_i}$ for $i\leq n$.
	Then $\alpha_i$ is total for $i > 0$, and we have\[\alpha_{i+1}( \top\wedge 
	a_1\wedge\dots\wedge a_{n-i-1})\ap a_{n-i} = 
	\alpha_i(\top\wedge a_1\wedge\dots\wedge a_{n-i})\] for $0<i\leq n$, and 
	\[\alpha_{1}( \top\wedge a_1\wedge\dots\wedge a_{n-1})\ap a_{n} \ges
	\alpha(\top\wedge a_1\wedge\dots\wedge a_n).\] With $e=\alpha_n(\top)$ the 
	claim follows by iterating.
	
	Claims 1 and 2 together imply that $(A,\cdot)$ is a PCA.  To show that the
	induced DCO structure is isomorphic to $\daf$, we show that for every
	$\alpha\in\funa$ there exists an $e\in\pcaa$ with $\alpha\subs\phi_e$ and 
	vice versa. For $\alpha\in\funa$ and $a\in A$ we have
	\[
	\alpha(a)\simeq (\alpha\rho)(\top\wedge a)\les 
	@((\widetilde{\alpha\!\circ\!\rho})(\top)\wedge a)\simeq 
	(\widetilde{\alpha\!\circ\!\rho})(\top)\ap a,
	\]
	thus the required element is given by $(\widetilde{\alpha\!\circ\!\rho})(\top)$.
	Conversely, $\phi_e\in\funa$ for any $e\in A$ since it can be represented as
	$@\circ\tatu{\const_e}{\id}$.
\end{proof}
\begin{corollary}\label{cor:char-pca}
	Up to isomorphism, the DCOs induced by PCAs are characterized by the fact that they are cartesian, shallow, and functionally complete.
\end{corollary}
\begin{proof}
	Lemmas~\ref{lem:pca-shallow-cartesian} and \ref{lem:pca-funcompl} say that the DCOs induced by PCAs are cartesian, shallow and functionally complete. Conversely,  Lemma~\ref{lem:dco-pca} establishes that any DCO having the three properties arises up to isomorphism from a PCA structure on the same carrier set.
\end{proof} 
\begin{remark}
	As shown in~\cite[Corollary~4.10.7]{frey2013fibrational}, dropping the shallowness condition in the corollary yields a characterization of DCOs induced by \emph{inclusions of PCAs}, analogous to Hofstra's characterization of filtered OPCAs among BCOs~\cite[Proposition~6.6]{hofstra2006all}.
\end{remark}

\section{Partitioned assemblies and local categories}\label{sec:local-categories}

\begin{definition}
  \begin{enumerate}
  \item The \emph{total category} $\gcons\fifp$ of an indexed preorder $\fifp$ 
  has pairs
$(I\in\catset, \varphi\in\fifp(I))$ as objects, and functions $f:I\to J$
satisfying $\varphi\leq f^*\psi$ as morphisms from $(I,\varphi)$ to
$(J,\psi)$.
\item The category $\pasm\daf$ of \emph{partitioned assemblies} over a DCO
  $\daf$ is the total category $\gcons\famf\daf$ of its family
  fibration.
 \end{enumerate}
\end{definition}
This definition of partitioned assemblies generalizes the traditional one in that if $(\pcaa,\pzf_\pcaa)$ is a DCO induced by a PCA, then $\pasm(\pcaa,\pzf_\pcaa)$ is \emph{equal} to the category of partitioned assemblies over $\pcad$ as defined \eg in~\cite[Definition~2.8]{menni2002more} (the original definition \cite[before Proposition~2]{carboni1988categorical} is slightly different, but is easily seen to be equivalent).

The total category $\gcons\fifc$ of a \emph{cartesian} indexed preorder $\fifc$ has finite limits: binary products and a 
terminal object are given by
\begin{equation}
  (I,\varphi)\times(J,\psi) = (I\times J,\pi_1^*\varphi\wedge\pi_2^*\psi)
\qqtext{and}    1 = (1,\top),
\end{equation}
and an equalizer of $f,g:(I,\varphi)\to(J,\psi)$ is given by
$m:(U,m^*\varphi)\mono(I,\varphi)$ where $m:U\mono I$ is the equalizer of $f$
and $g$ in $\catset$\footnote{This is an instance  of the fact that the total category of a finite-limit fibration over a finite-limit base has finite limits~\cite[Theorem~8.5]{streicherfib}}.

Furthermore there is an adjunction
\begin{equation}
  \abs{-}\;\adj\;\nabla:\catset\to\gcons\fifc\qqtext{given by}\abs{(I,\varphi)}=I\qtext{and}\nabla(I)=(I,\top).
\end{equation}
The left adjoint $\abs{-}$ is obviously faithful, and if $\fifc$ is shallow
it coincides with the \emph{global sections functor}
\[\Gamma \speq (\gcons\fifc)(1,-) \;:\; \gcons\fifc\to\catset,\]
which makes $\gcons\fifc$ a \emph{well-pointed
  local category}:
\begin{definition}\label{def:local}
  A \emph{local category} is a locally small finite-limit category $\catc$ whose
  global sections functor
has a right adjoint $\nabla$. If $\Gamma$ is faithful, we call $\catc$ 
\emph{well-pointed local}.
\end{definition}
The adjunction $\Gamma\adj\nabla$ is a
reflection for every local category $\catc$, since
\[
  \Gamma\nabla I = \catc(1,\nabla I)\cong\catset(\Gamma 
  1,I)\cong\catset(1,I)\cong I
\]
for all sets $I$. Since both $\nabla$ and $\Gamma$ preserve
limits, $\Gamma\adj\nabla$ is in fact a \emph{localization} (finite-limit
preserving reflection), whence we can make sense of the sheaf theoretic terms
\emph{dense}, \emph{closed}, and \emph{separated}:
\begin{definition}
  Let $\catc$ be a local category.
  \begin{enumerate}
  \item 
    An arrow $f:B\to A$ in $\catc$ is called
    \begin{itemize}
    \item 
    \emph{dense} if
    $\Gamma f$ is bijective, and
    \item\emph{closed} if
$
\begin{tikzcd}[column sep = large, row sep = small]
  B \ar[d,"\eta_{B}"']\ar[r,"f"'] & A\ar[d,"\eta_{A}"]\\
\nabla\Gamma B\ar[r,"\nabla \Gamma f"]& \nabla\Gamma A
\end{tikzcd}
$
is a pullback.
    \end{itemize}
\item
  An object 
  $G\in\catc$ is called
  \begin{itemize}
\item \emph{separated}, if $\eta_G:G\to\nabla\Gamma G$ is monic,
  \item 
  \emph{generic}, if for every $C\in\catc$
  there exists a closed $f:C\to G$, and
\item
  \emph{discrete}, if it is right
orthogonal to all closed maps over surjections, \ie if for any span
$C\xleftarrow{e}D\xrightarrow{f} G$ with $e$ closed and $\Gamma e$ surjective
there exists a unique $g:C\to G$ with $ge=f$.
  \end{itemize}
  \end{enumerate}
\end{definition}
\begin{remarks}\label{rem:local}
	\begin{enumerate}
\item It is well known~\cite{cassidy1985reflective} that for every localization, the dense and closed maps form a \emph{stable reflective factorization system}, in particular the dense maps are stable under pullback and satisfy 3-for-2.
\item\label{rem:local-gendisc}
A morphism $f: (I,\varphi)\to(J,\psi)$ in the total category of a shallow cartesian 
indexed preorder $\fifc$ is dense if and only if $f$ is a bijection, and closed 
precisely if $\varphi\cong f^*\psi$. The object $(I,\varphi)$ is
respectively generic or discrete in $\gcons\fifc$ precisely if $\varphi$ is 
so as a predicate in $\fifc$.
\item 
As a left adjoint, $\Gamma$ preserves epimorphisms. Conversely, if
$f:B\to A$ is closed and $\Gamma f$ is surjective then by the axiom of
choice it has a section, which implies that $f$ is split epic since split
epimorphisms are stable under arbitrary functors and pullbacks. Thus in
presence of choice, discrete objects are precisely those which are right
orthogonal to closed epis.  In well-pointed local categories this is even
true without choice, since faithful functors reflect epis.
\item\label{item:menni} 
Local categories are instances of the notion of \emph{chaotic situation} introduced by Menni in~\cite{menni2002more}, from where we also adapted the concept of `generic object'. Menni's article is closely related to the present work in that it is concerned with the relationship between partitioned assemblies and realizability toposes, elaborating on prior (but later published) work~\cite{menni2003characterization} by the same author.
	\end{enumerate}
	
\end{remarks}
\begin{lemma}\label{lem:char-wp} 
  The following are equivalent for a local category $\catc$.
  \begin{enumerate}
  \item $\catc$ is well pointed.
  \item All dense maps in $\catc$ are monos.
  \item All objects are separated.
    \item (If $\catc$ has a generic object $G$) $G$ is separated.
  \end{enumerate}
\end{lemma}
  \begin{proof}
    1 implies 2 since faithful functors reflect monomorphisms, and 2 implies
    3 since the maps $\eta_A:A\to \nabla\Gamma A$ are dense. 3 and 1
	are equivalent since the unit of an adjunction is componentwise monic if and 
	only if the left adjoint is faithful~\cite[Theorem~IV.3-1]{maclanecwm}. 
	Finally, if $\catc$ has a separated generic object then all objects are 
	separated since monos are stable under pullback.
  \end{proof}
\newcommand{\gfc}{\gcons\fifc_\catc}
\begin{proposition}\label{prop:pasm-char}
  A category $\catc$ is equivalent to 
  partitioned
  assemblies over a shallow cartesian DCO precisely if it is
  well-pointed local and has a discrete generic object $G$.
  
\end{proposition}
\begin{proof}
  Given a shallow cartesian DCO $\daf$, we have seen that
  $\famf\daf$ is shallow and cartesian and that $\mu\in\famf\daf(A)$ is discrete
  and generic. Then the total category $\gcons\famf\daf$ is well-pointed local
  by the remarks before Definition~\ref{def:local}, and $(A,\mu)$ is a discrete
  generic object by Remark~\ref{rem:local}-\ref{rem:local-gendisc}.
    
  Conversely, assume that $\catc$ is well-pointed local with discrete generic
  object $G$. By Proposition~\ref{prop:dcogenpred},
  Remark~\ref{rem:local}-\ref{rem:local-gendisc}, and since $\famf(-)$ reflects
  shallowness and cartesianness, it is sufficient to exhibit a shallow and
  cartesian indexed preorder $\fifc_\catc$ such that $\gfc\simeq\catc$.
  Elements of $\fifc_\catc(I)$ are dense maps $U\mono\nabla I$, ordered by
  inclusion, and reindexing along $f$ is given by pullback along $\nabla f$.  It
  is clear that all $\fifc_\catc(I)$ have -- and all $f^*$ preserve -- finite
  meets, so to conclude that $\fifc_\catc$ is a cartesian indexed preorder it remains
  to verify that all $\fifc_\catc(I)$ are (essentially) small. To this end we
  embed $\fifc_\catc(I)$ into $\catset(I,\Gamma G)$ by sending every dense
  $U\stackrel{u}{\mono}A$ to
  \[\chi_u=(I\xrightarrow[\cong]{\ve_I\inv}\Gamma\nabla 
  I\xrightarrow[\cong]{(\Gamma
    u)\inv}\Gamma U\xrightarrow{\Gamma c} \Gamma G),\] where $c:U\to G$ is
  closed.
In the commutative diagram
\[
\begin{tikzcd}[column sep = large]
   &
U \ar[rr,"c"]
\ar[dr,"\eta_U",tail]
\ar[ld,"u"',tail]
&&
G\ar[d,"\eta_G",tail]
\\
\nabla I\ar{r}[]{\eta_{\nabla I}}[']{=\nabla(\ve_I\inv)}
&\nabla\Gamma\nabla I\ar[r,"(\nabla\Gamma u)\inv"']
&\nabla\Gamma U\ar[r,"\nabla\Gamma c"']
&\nabla\Gamma G
\end{tikzcd}
\]
the outer trapezoid is a pullback since the inner right one is and the two
lower left maps are isomorphisms. This shows that
$u\cong (\nabla\chi_u)^*\eta_G$, \ie we can reconstruct $u$ up to isomorphism
from $\chi_u$.

To see that $\fifc_\catc$ is shallow let $U\mono 1$ be dense. Then $\Gamma
U=1$, \ie $U$ has a point which implies $1\cong U$.

Finally we have $\gcons\fifc_\catc\simeq\catc$ since the assignments
$C\mapsto(\Gamma C,\eta_C)$ and $(I, U\mono \nabla I)\mapsto U$ extend to
functors $J:\catc\to\gfc$ and $K:\gfc\to\catc$ satisfying $KJ=\id_\catc$ and
$JK\cong\id_{\gfc}$.
\end{proof}

Recall from~\cite[Remark~3.2]{carboni2000locally} that a finite-limit category $\catc$ is called \emph{weakly locally cartesian closed} (w.l.c.c.) if for all arrows $b:B\to J$ and $ u:J\to I$ the presheaf
\[
(\catc/J)(u^*-,b)\;:\; (\catc/I)\op\to\catset
\]
is weakly representable (\ie covered by a representable presheaf). In this case 
we call any weakly representing object in $\catc/I$ a \emph{weak dependent 
product of $b$ along $u$}. A \emph{weak exponential} of $B,C\in\catc$ is an object which weakly represents the presheaf $\catc(-\times B,C)$. W.l.c.c.\ categories do in particular have weak exponentials, and so do all their slices.

The following result is similar to~\cite[Corollary~2]{robinson2001abstract}\footnote{
See also the longer draft version~\cite{demarchi2000abstract} containing full 
proofs.}.
\begin{proposition}\label{prop:wlccc-funcompl}
  A cartesian DCO $\daf$ is functionally complete if and only if 
  the category $\pasm\daf$ is w.l.c.c.
\end{proposition}
\begin{proof}
  Assume first that $\daf$ is functionally  complete. A weak dependent product of $b:(B,\psi)\to(J,\varphi)$ along
  $u:(J,\varphi)\to (I,\iota)$ is given by $p:(K,\theta)\to(I,\iota)$, where
  \begin{align}
K &=\setof{(i\in I,\,f\in\textstyle\prod_{j\in J_i}B_j,\, a\in A)}{\forall j\in J_i\qdot @(a\wedge \varphi j)=\psi(fj)}\\
\theta(i,f,a)&=\iota(i)\wedge a\quad\\
\text{and}\quad p(i,f,a)&=i.
\end{align}

Conversely assume that $\pasm\daf$ is w.l.c.c.,
let $E=\setof{(\alpha,a,b)}{\alpha:A\pto A,f(a)=b}$, and let $F=(A\pto A)$ be 
the set of partial endofunctions on $A$. Define 
$\varphi,\psi: E\to A$ by $\varphi{(\alpha,a,b)}=a$
and $\psi{(\alpha,a,b)}=b$, respectively. 
Let $e_1:(E,\varphi)\to \nabla F$
and $e_2:(E,\psi)\to \nabla F$ be defined by 
$e_1(\alpha,a,b)=e_2(\alpha,a,b)=\alpha$. Let $f_X:(X,\xi)\to\nabla F$ 
together with $\ve:(X,\xi)\times_{\nabla F} (E,\varphi)\to (E,\psi)$ be a weak
exponential of $e_1$ and $e_2$ in 
$\famf\daf/\nabla F$, where
\begin{align}
  (X,\xi)\times_{\nabla F} (E,\varphi)&=(X\times_F E,\xi\wedge_F\varphi)\quad  
  \text{with}\\
(\xi\wedge_F\varphi)(x,e)&=\xi(x)\wedge\varphi(e).
  \end{align}
Let $@$ be a realizer of $\ve$, so that
$\psi\circ\ve=@\circ(\xi\wedge_F\varphi)$.
To show that $@$ is a universal function let $\alpha\in {\funa}$
and define
$f:(A,\id)\to\nabla F$ by $f(a)=\alpha(a\wedge-)$. A pullback of $f$
and $e_1$ is given by 
\begin{equation*}
\vcenter{\xymatrix{
(S,\theta)\pullbackcorner \ar[r]_k\ar[d]_h & (E,\varphi)\ar[d]^{e_1}\\
(A,\id)\ar[r]^{f}& \nabla F
}}
\quad
\begin{array}{r@{\;}l@{\qquad}l@{\;}l}
{S} = \{(a,b,c)&\msep \alpha(a\wedge b)=c\}\\
\theta{(a,b,c)} &= a\wedge b\\
h(a,b,c) &= a&\\
k(a,b,c) &= (f(a),b,c)&
\end{array}
\end{equation*}
where $h$ and $k$ are realized by $\lambda$ and $\rho$, respectively.  Define
$g:(S,\theta)\to (E,\psi)$ with the same underlying function as $k$.  Then $g$
is realized by $\alpha$. Let $\tilde g:(A,\id)\to (X,\xi)$ be a weak
exponential transpose, and let $\widetilde{\alpha}$ be a realizer of
$\tilde g$. The relevant data is summarized in the following diagram over 
$\nabla F$
\[\vcenter{
\xymatrix@R-2mm@C+3mm{
(A,\id)
  \ar[r]^{\tilde g}_{\widetilde{\alpha}} & 
(X,\xi)\\
(S,\theta)
  \ar[u]^{h}_\lambda
  \ar[r]^-{\tilde g\times_F\id}
  \ar[d]_k^\rho & 
(X\!\times_F\! E,\,\xi\!\wedge_F\!\varphi)
  \ar[u]^{\pi_1}_\lambda
  \ar[r]^-\ve_-{\att}
  \ar[d]_{\pi_2}^\rho 
& 
(E,\psi) \\
(E,\varphi) \ar[r]^\id &
(E,\varphi)
}},
\]
where the left/upper labels of arrows are functions, and the right/lower labels
are their realizers. We have
\begin{align*}
\psi\circ g & = \psi\circ\ve\circ(\tilde g\times_F\id)\\
& = @\circ(\xi\wedge_F\varphi)\circ(\tilde g\times_F\id)\\
& = @\circ\wedge\circ\langle\xi\circ\tilde g\circ h,\varphi\circ k\rangle
\end{align*}
and hence
\[
c 
= \psi(g(a,b,c)) 
= @((\xi\circ\tilde g\circ h)(a,b,c)\wedge(\varphi\circ k)(a,b,c))
= @(\widetilde{\alpha}(a)\wedge b)
\]
for $(a,b,c)\in S$, which means precisely that 
$\alpha(a\wedge b)\les @(\widetilde{\alpha}(a)\wedge b)$ for $a,b\in A$. 
\end{proof}
Combining the previous results, we get the following.
\begin{theorem}\label{thm:char-pasm}
  A category is equivalent to partitioned assemblies over a PCA if and only if 
  it is \wlcc and well-pointed local, and has a discrete generic object.
\end{theorem}
\begin{proof}
  By Lemmas~\ref{lem:pca-funcompl} and~\ref{lem:dco-pca}, PCAs can be identified 
  with functionally complete shallow cartesian DCOs. 
  Proposition~\ref{prop:pasm-char} establishes a correspondence between shallow 
  cartesian DCOs and well-pointed local categories with discrete generic 
  objects, and by Proposition~\ref{prop:wlccc-funcompl} a shallow cartesian DCO 
  is functionally complete if and only if the corresponding local category is 
  \wlcc
\end{proof}
Lemma~\ref{lem:char-wp} gives the following reformulation.
\begin{corollary}\label{cor:char-pasm}
	A category is equivalent to partitioned assemblies over a PCA if and only 
	if it is \wlcc and local, and has a \emph{separated} discrete generic object. \qed
\end{corollary}

\section{Characterizing realizability toposes}

\newcommand{\catcex}{\exlex{\catc}}

In this section we derive a characterization of realizability toposes from
Corollary~\ref{cor:char-pasm} and the fact that realizability toposes are exact 
completions of partitioned assemblies. We start by recalling relevant facts 
about exact completion.

An \emph{exact} category is a finite-limit category with
pullback-stable regular-epi/mono factorizations, in which every equivalence
relation is a kernel pair.
Exact categories form a 2-category $\catex$ (1-cells
are functors preserving finite limits and regular epis), and the inclusion
functor $ \catex\incl\catlex$ from exact into finite-limit (`left exact')
categories has a left biadjoint
\begin{equation}\label{eq:biadj}
 \exlex{(-)}\,:\,\catlex\to\catex
 \end{equation}
known as \emph{exact completion}~\cite{carboni1982free}. 

An object $P$ in an exact category $\catx$ is called \emph{(regular) 
projective} if for every regular epimorphism $e:Y\epi X$ and every $f:P\to X$ 
there exists a $g:P\to Y$ with $eg=f$ 
\[
\begin{tikzcd}[row sep = small,column sep = small]
&	P 	\ar[rd,"f"]
		\ar[ld,dashed,"g"']
\\ 	Y 	\ar[rr,"e"',two heads]
&&	X
\end{tikzcd}	
\]
(we will drop the `regular' and simply say `projective'). Using this concept, we 
can characterize exact
completions\footnote{The first reference that I could find for 
Theorem~\ref{thm:exact-completion} is 
Robinson and Rosolini's~\cite[Proposition~4.1]{robinson1990colimit}. There it is 
attributed to Joyal, Carboni, and Celia-Magno, but on my inquiry Rosolini told 
me that they discovered it themselves and learned later from Carboni that he 
already knew.}.
\begin{theorem}\label{thm:exact-completion}
\begin{enumerate}
\item\label{thm:exact-completion-proj-image} For any finite-limit category
$\catc$, the unit functor $\catc\to\exlex{\catc}$ is full and faithful, and its 
essential image coincides with the full subcategory $\proj(\exlex{\catc})$ of 
$\exlex{\catc}$ on projective objects.
\item\label{thm:exact-completion-characterize} An exact category $\catx$ is 
(equivalent to) an exact completion if and only if 
it has \emph{enough projectives} (i.e.\ every object $X$ can be covered 
by a projective
$P$  via a regular epi $P\epi X$), and $\proj(\catx)$ is closed
under finite limits in $\catx$.\hfill${}_\blacksquare$
\end{enumerate}
\end{theorem}
We observe that being an exact completion of a finite-limit 
category is a \emph{property} of an exact category rather than additional 
structure. 
In particular, realizability toposes are exact completions:
\begin{theorem}\label{thm:rt-ex-comp}
	The realizability topos $\catrt\pcad$ over a PCA $\pcad$ is equivalent to
	the exact completion of the category $\pasm\pcad$ of partitioned assemblies
	over $\pcad$. \qed
\end{theorem}
This result depends on the axiom of choice.
As mentioned in the introduction it was shown by Robinson and 
Rosolini~\cite[2.2 and 4.2]{robinson1990colimit} for the effective topos. The 
generalization
to arbitrary PCAs is straightforward; the statement can be found \eg 
in~\cite[Section~2.3]{hofstra2003ordered} in even greater generality for OPCAs.

Combining Corollary~\ref{cor:char-pasm} with the preceding theorems immediately 
gives us the following.
\begin{proposition}\label{prop:char-rt}
	An exact category $\catx$ is equivalent to a realizability topos over a PCA 
	precisely if it has enough projectives, 
	projective objects are closed under finite limits in $\catx$,	
	and $\proj(\catx)$ is a \wlcc local 
	category with a separated discrete generic object.
\end{proposition}
\begin{proof}
	Realizability toposes $\catrt\pcad$ are exact completions of
	$\pasm\pcad$ by Theorem~\ref{thm:rt-ex-comp}, and hence have enough 
	projectives which are closed under finite limits by
	Theorem~\ref{thm:exact-completion}-\ref{thm:exact-completion-characterize}.
	By Theorem~\ref{thm:exact-completion}-\ref{thm:exact-completion-proj-image} 
	we 	have $\proj(\catx)\simeq\pasm\pcad$, and the latter category is \wlcc 
	local and has a separated discrete generic object by 
	Corollary~\ref{cor:char-pasm}.

	Conversely, assume that $\catx$ is exact, has enough projectives which are
	closed under finite limits, and $\proj(\catx)$ is \wlcc local with separated
	discrete generic object. Then $\proj(\catx)$ is a category of partitioned
	assemblies by Corollary~\ref{cor:char-pasm}, $\catx$ is its exact completion
	by Theorem~\ref{thm:exact-completion}, and is thus equivalent to 
	$\catrt\pcad$ by Theorem~\ref{thm:rt-ex-comp}.
\end{proof}

In the following we give a more elementary rephrasing of this result, by 
reformulating the conditions on $\proj(\catx)$ directly as conditions on 
$\catx$. For this we need the following lemma on exact completion of local categories.
\begin{lemma}\label{lem:ex-nabla}
	If $\catx$ is an exact completion and $\proj(\catx)$ is local, then $\catx$ 
	is local as well and the right adjoint to $\catx(1,-)$ is given the right 
	adjoint to $\proj(\catx)(1,-)$ composed with the inclusion 
	$\proj(\catx)\incl\catx$.
	\[\begin{tikzcd}
		\catset\ar[r,"\nabla"']\ar[rd,"\nabla"'] & \proj(\catx)\ar[d,hook]\ar[rd,"\Gamma"]\\
		& \catx\ar[r,"\Gamma"] & \catset
	\end{tikzcd}
	\]
\end{lemma}
\begin{proof}
Exact completion preserves local smallness
since every hom-set $\catx(X,Y)$ can be presented as subquotient of $\catx(P,Q)$ for projective covers $P$ and $Q$ of $X$ and $Y$, respectively.

To see that $\catx(1,-)$ has the claimed right adjoint, let $X\in\catx$ and let $P\epi X$ be a projective cover of $X$. Covering the kernel of $e$ by another projective $Q$ we can represent $X$ as a coequalizer 
 \[Q\xrightrightarrows[g]{f}P\xepi{e}X.\]
 The functor $\Gamma:\catx\to\catset$ is regular since $1$ is projective in 
 $\catx$, hence
 \[
 \Gamma Q\xrightrightarrows[\Gamma g]{\Gamma f}\Gamma P\xepi{\Gamma e}\Gamma X
 \]
 is a coequalizer as well. For $I\in\catset$ we have
 \begin{align*}
 \catx(X,\nabla I)&\cong\setof{h:P\to\nabla I}{h\circ f=h\circ g}\\
 &\cong\setof{k:\Gamma P\to I}{k\circ\Gamma f=k\circ\Gamma g}\cong
 \catset(\Gamma X,I)
 \end{align*}
 which means that $\catset(\Gamma-,I):\catx\op\to\catset$ is 
 represented by $\nabla I$.
\end{proof}
Thus, if $\catx$ is an exact completion of a local category then the 
localization on $\proj(\catx)$ is the restriction of the localization on 
$\catx$, which means in particular that morphisms/objects in $\proj(\catx)$ are 
closed/dense/separated in $\proj(\catx)$ if and only if they are so in $\catx$. 
The next lemma establishes the same fact for \emph{discreteness}.
\begin{lemma}\label{lem:preserve-discrete}
Let $\catx$ be an exact completion of a local category.
A projective object $D$ is discrete in $\catx$ if and only if it is discrete in 
$\proj(\catx)$.
\end{lemma}
\begin{proof}
	Clearly $D$ is discrete in $\proj(\catx)$ whenever it is in $\catx$. 
	Conversely assume that $D$ is discrete in $\proj(\catx)$, and let 
	$u:Y\epi X$ in $\catx$ be closed with $\Gamma u$ surjective. Then $\Gamma u$
	splits by the axiom of choice, and $u$ splits since split epics are stable 
	under
	functors and pullbacks. Thus $u$ is in particular regular epic.
	Let $e:P\epi X$ be a projective cover of $X$, and consider the following 
	diagram.
	\[
	\xymatrix{
		Q
		\depi[r]^p
		\depi[d]_v
		\pullbackcorner 
		\emar[dr]|{(\ast)}
		& Y
		\ar[r]^{\eta_Y}
		\depi[d]_u
		\pullbackcorner 
		& \nabla\Gamma Y \depi[d]^{\nabla\Gamma u}
		\\ P 
		\depi[r]_e
		& X 
		\pullbackcorner[lu]
		\ar[r]_{\eta_X}
		& \nabla\Gamma X
	}
	\]
	The square ($\ast$) is the pullback of $u$ and $e$, and since both are 
	regular	epis it is also a pushout (easy exercise in regular categories, 
	compare and contrast with the fact that a pullback square in an 
	\emph{abelian} category is a pushout already if one of the legs is an 
	epi~\cite[Section~2.5]{freyd2003abelian}). $Q$ is projective as pullback of 
	the outer square, and $v$ is closed with surjective $\Gamma v$ since the 
	same is true for $u$. Now let $f:Y\to D$. By discreteness of $D$ in 
	projectives there exists a $g:P\to D$ with $fp=gv$, and since ($\ast$) is a 
	pushout there exists an $h:X\to D$ with $hu=f$ and $he=g$.
\end{proof}
The preceding lemmas together with Carboni and Rosolini's result~\cite
{carboni2000locally} that an exact completion $\catx$ is locally cartesian 
closed precisely if $\proj(\catx)$ is \wlcc yield the following reformulation of 
Proposition~\ref{prop:char-rt}.
\begin{theorem}\label{thm:char-rt'}
A locally small category $\catx$ is equivalent to a realizability topos over a 
PCA precisely if
\begin{enumerate}
\item $\catx$ is exact and locally cartesian closed,
\item $\catx$ has enough projectives and projective objects are closed under 
finite limits in $\catx$,
\item $\Gamma:\catx\to\catset$ has a right adjoint factoring through 
$\proj(\catx)\incl\catx$, and
\item there is a discrete and separated projective object $G$ admitting a closed 
map from any other projective.
\end{enumerate}
\qed
\end{theorem}

\subsection*{Acknowledgements}

The present work is based on the author's PhD thesis~\cite{frey2013fibrational}.
Thanks to my supervisor Paul-André Melliès for support and encouragement, to 
Thomas Streicher for long email exchanges, to Pino Rosolini for discussions 
about exact completion, and to the anonymous referees for helpful feedback.

The author acknowledges partial support  by the ERC Advanced Grant ECSYM, the 
Danish Council for Independent Research Sapere Aude grant ``Complexity via Logic 
and Algebra'' (COLA), and the U.S. Air Force Office of Scientific Research MURI 
grant FA9550-15-1-0053.

\nocite{lietz2002impredicativity}

\bibliographystyle{amsalpha}
\bibliography{../shared/bib.bib}
\end{document}